\documentclass[12pt,a4paper]{amsart}
\usepackage{geometry,amsrefs}
\usepackage{amssymb}

\newtheorem{theo}{Theorem}
\newtheorem{cor}{Corollary}
\newcommand{\I}{\mathrm{i}}
\newcommand{\Ad}{\operatorname{Ad}}
\newcommand{\id}{\operatorname{id}}
\newcommand{\Span}{\operatorname{span}}
\newcommand{\C}{\mathbb{C}}
\newcommand{\R}{\mathbb{R}}
\newcommand{\Lie}[1]{\mathfrak{#1}}
\usepackage{xcolor}

\title{CR embeddings of CR manifolds}

\author{M. G. Cowling}
\address{School of Mathematics and Statistics\\
University of New South Wales\\
UNSW Sydney NSW 2052\\
Australia}
\email{m.cowling@unsw.edu.au}
\author{M. Ganji}
\address{School of Science and Technology\\
University of New England\\
Armidale NSW 2351\\
Australia}
\email{mganjia2@une.edu.au}
\author{A. Ottazzi}
\address{School of Mathematics and Statistics\\
University of New South Wales\\
UNSW Sydney NSW 2052\\
Australia}
\email{a.ottazzi@unsw.edu.au}
\author{G. Schmalz}
\address{School of Science and Technology \\
University of New England \\
Armidale NSW 2351 \\
Australia}
\email{schmalz@une.edu.au}
\thanks{The first-named author was supported by the Australian Research Council grant DP170103025}
\subjclass[2020]{32V30}
\keywords{CR manifold, CR embedding}

\begin{document}
\begin{abstract}
We improve results of Baouendi, Rothschild and Treves and of Hill and Nacinovich by finding a much weaker sufficient condition for a CR manifold of type $(n,k)$ to admit a local CR embedding into a CR manifold of type $(n+\ell,k-\ell)$.
While their results require the existence of a finite dimensional solvable transverse Lie algebra of vector fields; we require only a finite dimensional extension.
\end{abstract}
\maketitle

\section{Introduction and Notation}

Consider a CR manifold $(M,D,J)$ of type $(n,k)$.
This means that $M$ is a manifold of dimension $2n+k$ with a rank $2n$ distribution $D$ and a field of endomorphisms $J\colon D\to D$ such that $J^2=-\id$.
We assume that $M$ is integrable, meaning that the $-\I$ eigendistribution $D^{0,1}$ of $J$ is involutive, that is,
\[
[\Lie{X}^{0,1} , \Lie{X}^{0,1} ]\subseteq \Lie{X}^{0,1} ,
\]
where
\[
\Lie{X}^{0,1}
= \{Z\in \Gamma(D)_\C \colon Z=X+\I JX, X\in \Gamma(D)  \}.
\]
As usual, $\Gamma$ indicates a space of sections and a subscript $\C$ indicates a complexification.
We write $\Lie{X}$ for the space of vector fields on $M$.

We say that $F\colon M\to \tilde{M}$ is a CR embedding of a CR-manifold $(M,D,J)$ of type $(n,k)$ into another CR-manifold $(\tilde{M},\tilde{D},\tilde{J})$ of type $(n+\ell,k-\ell)$ if $F$ is a smooth embedding and $F_*\colon D^{0,1} \to  \tilde{D}^{0,1}$ satisfies
\[
F_* J = \tilde J F_*.
\]
The case where $\ell = k$ is of particular interest, as this corresponds to an embedding into a complex manifold.

Finding embeddings of the type envisaged above is one of the fundamental problems in CR geometry.
We are going to consider the local problem only; that is, we fix a point $p$ and look for an embedding of a neighbourhood of $p$.
This means that we can replace $M$ by a small open neighbourhood of $p$ at any stage.

We mention a few contributions that are particularly relevant.
It is well known that analytic CR-manifolds can always be locally embedded in complex space.
Baouendi, Rothschild and Treves \cite{MR809720} consider the case where there is an abelian Lie algebra $\Lie{g}$ of real vector fields that is transverse, in the sense that
\[
\Lie{X} = \Gamma(D) \oplus \Lie{g} ,
\]
and normalising, in the sense that
\begin{equation}\label{eq:normalise}
[\Lie{g}, \Lie{X}^{0,1}] \subseteq  \Lie{X}^{0,1},
\end{equation}
and construct an embedding into a complex space.
Jacobowitz \cite{MR876018} considers the case where $\ell=1$ and finds a condition for the existence of an embedding into $\C^{n+1}$.
Finally, Hill and Nacinovich \cite{MR1794546} treat the case where there is a solvable transverse normalising Lie algebra of complex vector fields of dimension $\ell$, and construct an embedding into a manifold of type $(n+\ell, k-\ell)$; they use solvability to extend by induction on dimension.
We are going to treat the case of a finite-dimensional Lie algebra extension of $\Lie{X}^{0,1}$ in $\Lie{X}_{\C}$ by nonvanishing complex vector fields $X_1$, \dots, $X_s$, and show that $M$ embeds into a CR manifold $\tilde{M}$ of type $(n+\ell, k-\ell)$ if $\dim(\Lie{X}^{(0,1)} + \langle X_1, \dots, X_s\rangle) = n+\ell$.

\section{Main results}
We state our theorem more precisely.

\begin{theo}
Let $(M, D, J)$ be a CR-manifold of type $(n,k)$.
Suppose that $X_1,\dots, X_s$ are nonvanishing complex vector fields that normalise $\Lie{X}^{0,1}$ (as in \eqref{eq:normalise}) and satisfy
\[
[X_\alpha,X_\beta]=c_{\alpha\beta}^\gamma X_\gamma \mod  {\Lie{X}}^{0,1},
\]
where the $c_{\alpha\beta}^\gamma$ are constants.
If
\[
\dim( \Lie{X}^{0,1} + \Span\{ X_1, \dots, X_s \} ) = n + \ell,
\]
then there is a (local) CR-embedding of $M$ into a CR-manifold $\tilde{M}$ of type $(n+\ell,k-\ell)$.
\end{theo}

\begin{proof}
Fix $p \in M$.
Without loss of generality we assume that each $X_\alpha(p)$ is not purely imaginary.
The $c_{\alpha\beta}^\gamma$ are the structure constants of the Lie algebra $\Lie{g}$ defined by
\[
\Lie{g} := (\Span\{ X_1, \dots, X_s \} + \Lie{X}^{0,1})/\Lie{X}^{0,1}.
\]

By renumbering the vector fields and passing to a submanifold of $M$ containing $p$ if necessary, we may suppose that
\[
\Lie{X}^{0,1} + \Span\{ X_1, \dots, X_s \}
= \Lie{X}^{0,1} \oplus \Span\{ X_1, \dots, X_\ell \} ,
\]
where $\ell\le s$.
We shall construct complex vector fields
\[
Y_\alpha=\lambda_\alpha^\gamma(t) X_\gamma+  \I \partial_{\alpha},
\]
where $\alpha=1,\dots, s$ on a neighbourhood of $(p,0)$ in $M\times\R^s$ such that
\begin{equation}\label{abel}
[Y_\alpha, Y_\beta] = 0 \mod \Lie{X}^{0,1};
\end{equation}
here $t^1,\dots,t^s$ are coordinates in $\R^s$ and $\partial_{\alpha}$ means ${\partial}/{\partial t^\alpha}$.
Then we shall show that the functions $\lambda_\alpha^\gamma$ can be chosen in such a way that
\begin{equation}\label{eq:depends}
\lambda_\alpha^\gamma(t^1, \dots, t^s)
= \lambda_\alpha^\gamma (t^1, \dots, t^\ell)
\end{equation}
when $\alpha\le \ell$.
It will then follow quickly that the vector fields $Y_\alpha$ with $\alpha\le \ell$ define a CR structure on $M\times V$, where $V$ is a neighbourhood of $0$ in $\R^\ell$, and there is a CR-embedding of $M$ in $M\times V$.

To show that \eqref{abel} holds, we observe that
the $Y_\alpha$ preserve the (lifted) $\Lie{X}^{0,1}$, and choose the functions $\lambda_{\alpha}^\gamma(t)$ such that $\lambda_\alpha^\gamma(0) = \delta_\alpha^\gamma$ and the $Y_\alpha$ commute modulo $\Lie{X}^{0,1}$.
Equivalently,
\begin{equation}\label{eq1}
\partial _\alpha\lambda^\gamma_\beta-\partial _\beta\lambda^\gamma_\alpha=\I \lambda_\alpha^\mu\lambda_\beta^\nu c_{\mu\nu}^\gamma.
\end{equation}

Consider this system of PDE.
Let $\{\xi_\gamma\}$ be a basis of an abstract copy of the Lie algebra $\Lie{g}$ and $(t^1,\dots, t^s)$ be coordinates of $\Lie{g}$ with respect to this basis.
Then $\Lambda := \lambda^\gamma_\alpha  \xi_\gamma dt^\alpha$ is a Lie algebra valued $1$-form, and the system \eqref{eq1} may be rewritten as
\begin{equation}\label{as}
d\Lambda=\frac{\I}{2} [\Lambda,\Lambda],
\end{equation}
where $d$ is the exterior derivative with respect to the $t$ variables.
This nonlinear autonomous system of PDE is similar to the structure equation of the Maurer--Cartan form, and this similarity allows us to solve the system \eqref{as}.
Let $\Omega$ be the left-invariant Maurer--Cartan form on the simply connected Lie group $G$ with Lie algebra $\Lie{g}$.
Then $\Omega$ satisfies the Maurer--Cartan equation
\[
d\Omega=-\frac12 [\Omega,\Omega].
\]
Let $t$ be real-analytic local coordinates on a neighbourhood of the identity $e$ in $G$ such that $0$ corresponds to $e$ and define $\Omega :=\omega^\gamma_\alpha(t) \xi_\gamma dt^\alpha$.
Then $\Omega(0)=\xi_\alpha dt^\alpha$.
Let
\[
\lambda^\gamma_\alpha(t)=\omega^\gamma_\alpha(-\I t).
\]
This is well defined since the $\omega^\gamma_\alpha$ are real-analytic, and the $\lambda^\gamma_\alpha$ defined in this way satisfy the equations \eqref{eq1} and $\omega^\gamma_\alpha(0)=\delta^\gamma_\alpha$.

To arrange that \eqref{eq:depends} holds, we suppose that $t^1$, \dots, $t^s$ are exponential coordinates of the second kind in some neighbourhood of $e$ in $G$, that is,
\[
g=\exp(t^{s}\xi_{s})\dots \exp(t^{1}\xi_1).
\]
We observe that the $dt^\alpha$ component of the Maurer--Cartan form depends on $t^1$, \dots, $t^{\alpha-1}$ only.
Indeed, the (left-invariant) Maurer--Cartan form is
\[
dL_{g^{-1}} dg
\]
and the $dt^\alpha$ component is
\begin{align*}
dL_{g^{-1}} \frac{\partial g}{\partial t^\alpha}
&=dL_{\exp(-t_1\xi_1)}\dots dL_{\exp(-t^{s}\xi_{s})} \\
& \qquad\qquad\frac{\partial}{\partial t_\alpha} L_{\exp(t^{s}\xi_{s})}\dots L_{\exp(t^{\alpha+1}\xi_{\alpha+1})} R_{\exp(t^{1}\xi_{1})}\dots R_{\exp(t^{\alpha-1}\xi_{\alpha-1}) }  \exp(t^{\alpha}\xi_{\alpha})\\
&=dL_{\exp(-t^1\xi_1)}\dots dL_{\exp(-t^{\alpha-1}\xi_{\alpha-1})}dR_{\exp(t^{1}\xi_{1})}\dots dR_{\exp(t^{\alpha-1}\xi_{\alpha-1}) }\xi_\alpha\\
&=\Ad_{\exp(-t^1\xi_1)}\dots \Ad_{\exp(-t^{\alpha-1}\xi_{\alpha-1})}\xi_\alpha.
\end{align*}
Here $L$ and $R$ denote left and right translations.
Therefore the functions $\lambda_{\alpha}^\gamma$ do indeed depend only on the variables $t^\mu$ with $\mu<\alpha$.

It follows that $\Lie{X}^{0,1} \oplus \langle Y_1,\dots, Y_\ell\rangle$ is well defined on $M\times V$, where $V$ is a suitable neighbourhood of $0$ in $\R^\ell$.
It remains to show that $\Lie{Y}^{0,1}$, the span of (the lift of) $\Lie{X}^{0,1}$ and the vector fields $Y_1,\dots,Y_\ell$ defines a CR structure on $\tilde{M} = M \times V$, that is,
\[
\Lie{Y}^{0,1}\cap\overline{\Lie{Y}^{0,1}} = \{0\}.
\]

Suppose that $V+a^j Y_j \in \Lie{Y}^{0,1}$ and  $W+b^k \overline{Y}_k \in \overline{\Lie{Y}^{0,1}}$.
If
\[
V+a^j Y_j=W+b^k \overline{Y}_k,
\]
then
\[
W-V=a^j Y_j - b^k \overline{Y}_k=0,
\]
that is $V=W=0$, and also
\[
a^j (X_j+\I \partial_j)-b^j (\overline{X}_j-\I\partial_j)=0.
\]
Therefore $a_j=-b_j$ and hence
\[
a_j (X_j+\overline{X_j})=0.
\]
Since $X_j(p)$ is not purely imaginary and the $X_j(p)$ are linearly independent, $X_j + \overline{X}_j \neq 0$ in a neighbourhood of $p$, and so (again passing to a submanifold if necessary) $a_j=0$ is the only solution.
\end{proof}

It may be worth remarking that if the CR manifold $(M,D,J)$ admits a CR embedding into a complex space, then in fact the conditions of Baouendi, Rothschild and Treves \cite{MR809720} are satisfied, and \emph{a fortiori} those of Hill and Nacinovich \cite{MR1794546}, and ours too.
It is less clear what happens when there is a CR embedding into another CR manifold that is not a complex space.

It may also be helpful to note that in the special case where the vector fields $X_1$, \dots, $X_s$ are real, then they generate (local) flows that preserve the CR structure; if they also generate a Lie algebra, then this generates a (local) group of transformations that preserves the structure.
Further, even in the more special case where there is a transverse normalising Lie algebra of real vector fields, then our result extends that of \cite{MR809720}; in this case, the use of exponential coordinates of the second kind is not necessary.

Here is a corollary of the proof of our theorem.

\begin{cor}
Let $(G,D,J)$ be a left-invariant CR structure on a Lie group $G$.
Then $G$ can be locally embedded into complex space.
\end{cor}

\begin{proof}
Let $\{X_1,\dots,X_{s}\}$ be a basis of right-invariant vector fields such that $X_1$, \dots, $X_\ell$ are transverse to $D$ at $e\in G$.

As before we can find (complex) functions $\lambda_{\alpha\beta}(t)$ such that the vector fields $Y_\alpha$, given by
\[
Y_\alpha
:= \sum \lambda_{\alpha\beta}(t) X_\beta + \I \frac{\partial}{\partial t_\alpha},
\]
commute, and let $t^1,\dots,t^s$ be exponential coordinates of the second kind on $G$.
Then $\Lie{X}^{0,1} +\langle Y_1,\dots, Y_\ell\rangle$ determines an integrable complex structure on (a neighbourhood of $e \times 0$ in) $G\times \R^\ell$.
\end{proof}

Of course, this was already known, as everything is analytic in this case, but arguably this proof is simpler.

\end{document}